\documentclass[12pt,a4paper]{amsart}
\usepackage[top=40mm, bottom=40mm, left=30mm, right=30mm]{geometry}
\usepackage{mathptmx}
\usepackage{graphicx}
\usepackage{amsthm}
\usepackage{amssymb}
\usepackage{color,xcolor}
\usepackage[colorlinks=true,citecolor=blue]{hyperref}
\usepackage[all,cmtip]{xy}
\usepackage{mathtools}

\usepackage[nameinlink]{cleveref}
\crefname{theorem}{Theorem}{Theorems}
\crefname{Maintheorem}{Theorem}{Theorems}
\crefname{lemma}{Lemma}{Lemmas}
\crefname{remark}{Remark}{Remarks}
\crefname{prop}{Proposition}{Propositions}
\crefname{definition}{Definition}{Definitions}
\crefname{corollary}{Corollary}{Corollaries}
\crefname{cor}{Corollary}{Corollaries}
\crefname{section}{Section}{Sections}
\crefname{subsection}{Subsection}{Sections}
\crefname{figure}{Figure}{Figures}
\crefname{quest}{Question}{Questions}
\crefname{claim}{Claim}{Claims}
\crefname{case}{Case}{Cases}
\crefname{Maincase}{Case}{Cases}
\usepackage{enumitem}
\usepackage{pstricks}
\usepackage{tikz}
 \usepackage{diagbox}
 \usepackage[title]{appendix}
\newtheorem{theorem}{Theorem}[section]

\newcounter{MainTheoremCounter}

\newtheorem{Maintheorem}[MainTheoremCounter]{Theorem}

\theoremstyle{definition}
\newtheorem{definition}[theorem]{Definition}

\newtheorem{cor}[theorem]{Corollary}

\numberwithin{equation}{section}

\begin{document}\title{Veech's theorem of higher order}

\author[J.~Qiu]{Jiahao Qiu}

\author[X.~Ye]{Xiangdong Ye}
\address[J.~Qiu, X.~Ye]{Wu Wen-Tsun Key Laboratory of Mathematics, USTC, Chinese Academy of Sciences and
School of Mathematics, University of Science and Technology of China,
Hefei, Anhui, 230026, P.R. China}
\email{qiujh@mail.ustc.edu.cn}
\email{yexd@ustc.edu.cn}

\date{\today}

\begin{abstract}
For an abelian group $G$,
 $\vec{g}=(g_1,\ldots,g_d)\in G^d$ and $\epsilon=(\epsilon(1),\ldots,\epsilon(d))\in \{0,1\}^d$,
let
$\vec{g}\cdot \epsilon=\prod_{i=1}^{d}g_i^{\epsilon(i)}$.
In this paper,
it is shown that
for a minimal system $(X,G)$ with $G$ being abelian,
$(x,y)\in \mathbf{RP}^{[d]}$ if and only if there exists a sequence
$\{\vec{g}_n\}_{n\in \mathbb{N}}\subseteq G^d$ and points $z_{\epsilon}\in X,\epsilon\in \{0,1\}^d$
with $z_{\vec{0}}=y$
such that for every $\epsilon\in \{0,1\}^d\backslash\{ \vec{0}\}$,
\[
\lim_{n\to\infty}(\vec{g}_n\cdot\epsilon)x= z_\epsilon\quad \mathrm{and} \quad
\lim_{n\to\infty}(\vec{g}_n\cdot\epsilon)^{-1}z_{\vec{1}}=z_{\vec{1}-\epsilon},
\]
where $\mathbf{RP}^{[d]}$ is the regionally proximal relation of order $d$.
\end{abstract}
\maketitle

\section{Introduction}

By a \emph{topological dynamical system}  (t.d.s. for short), 
we refer to a pair $(X, G)$, where $X$ is a compact metric space, and $G$ acts  on it as an abelian group of homeomorphisms.

\medskip

In a certain sense, 
an equicontinuous system represents the most fundamental structure within the realm of topological dynamical systems. 
The characterization of the equicontinuous structure relation $S_{\mathrm{eq}}(X)$ for a t.d.s.
$(X,G)$ is one of the earliest problems studied in this field; 
specifically, it involves identifying the smallest closed invariant equivalence relation 
$R(X)$ on $ (X,G)$, such that $ (X/R(X),G)$ is equicontinuous.
It was shown in \cite{EG60} that $ S_{\mathrm{eq}}(X) $ is the smallest closed invariant equivalence relation 
containing the regionally proximal relation $ \mathbf{RP} = \mathbf{RP}(X,G) $. 
Recall that $ (x,y) \in  \mathbf{RP} $ if there exist sequences $ \{x_n\}_{n\in \mathbb{N}} ,\{y_n\}_{n\in \mathbb{N}} \subseteq X, \{g_n\}_{n\in \mathbb{N}}\subseteq G  $ 
such that $\lim\limits_{n\to\infty}x_n= x,\lim\limits_{n\to\infty}y_n=  y$ and $\lim\limits_{n\to\infty}(g_nx_n ,g_ny_n )=  (z,z)$
for some $z \in X$.
Naturally, one might  ask whether $ S_{\mathrm{eq}} = \mathbf{RP} $ holds for any minimal system. 
Veech \cite{VAW68} was the first to provide a positive answer to this question; 
he showed that $ S_{\mathrm{eq}} = \mathbf{RP} $ is indeed valid for all minimal systems.
As a matter of fact, Veech proved
that for a minimal system $(X,G)$, $(x,y)\in  \mathbf{RP}$ if and only if there exists a sequence
$\{g_n\}_{n\in \mathbb{N}}\subseteq G$ and $z\in  X$ such that
\[
\lim\limits_{n\to\infty}g_nx= z
\quad
\mathrm{and}
\quad
\lim\limits_{n\to\infty}g_n^{-1}z= y.
\]

\medskip

Nilpotent structures derived from ergodic systems play a crucial role in the study of ergodic theory and its applications to combinatorial number theory. 
For further details, please refer to \cite{HKbook}.
A natural question arises regarding how to obtain analogous nilpotent structures in the context of topological dynamics.
 In a pioneering study, Host, Kra and Maass introduced the concept of the {\it regionally proximal relation of order $d$} for a t.d.s. $(X,G)$ in their work \cite{HKM}. 
 This notion is denoted by $\mathbf{RP}^{[d]}$ when $G=\mathbb{Z}$, and can be readily extended to any abelian group.
We observe that the regionally proximal relation of order $1$ is equivalent to the classical regionally proximal relation, that is, $\mathbf{RP}^{[1]} = \mathbf{RP}$. 
It is evident that $\mathbf{RP}^{[d]}$ constitutes a closed invariant relation.
For a minimal distal system, the authors in \cite{HKM}
proved  that $\mathbf{RP}^{[d]}$ constitutes an equivalence relation and that the quotient space $X/\mathbf{RP}^{[d]}$ is referred to as the $d$-step pro-nilsystem.
Subsequently, Shao and Ye \cite{SY12} showed that for any minimal system, $\mathbf{RP}^{[d]}$ indeed constitutes an equivalence relation. 
Moreover, $\mathbf{RP}^{[d]}$ possesses what is referred to as the lifting property.
The combined findings from \cite{HKM} and \cite{SY12} indicate that for any minimal system $(X,\mathbb{Z})$, the quotient space $X/\mathbf{RP}^{[d]}$ forms a $d$-step pro-nilsystem. 
Note that the notion of regionally proximal relation of higher order
can be generalized to any topological group, see \cite{GGY} by Glasner, Gutman and Ye.

A systematic investigation into the properties of $\mathbf{RP}^{[d]}$ concerning $\mathbb{Z}$-actions was conducted by Huang, Shao and Ye in \cite{HSY}. 
An open question that remains is whether a characterization of $\mathbf{RP}^{[d]}$ analogous to the one established by Veech for $\mathbf{RP} = \mathbf{RP}^{[1]}$  in \cite{VAW68} can be obtained.
In this paper, we provide an affirmative  answer  to the question by extending Veech's result to the higher order for abelian groups. That is,

\begin{Maintheorem}\label{main-thm}
Let $(X,G)$ be a minimal system with $G$ being abelian and $d\in \mathbb{N}$.
Then $(x,y)\in \mathbf{RP}^{[d]}$ if and only if there exists a sequence
  $\{\vec{g}_n\}_{n\in \mathbb{N}}\subseteq G^d$ and
points $z_{\epsilon}\in X,\epsilon\in \{0,1\}^d$
with $z_{\vec{0}}=y$
  such that for every $\epsilon\in \{0,1\}^d\backslash\{ \vec{0}\}$,
  \[
 \lim_{n\to\infty} (\vec{g}_n\cdot\epsilon)x=z_\epsilon\quad \mathrm{and} \quad
    \lim_{n\to\infty}(\vec{g}_n\cdot\epsilon)^{-1}z_{\vec{1}}=z_{\vec{1}-\epsilon}.
  \]
\end{Maintheorem}

\bigskip

To enhance our understanding of the theorem, we will illustrate the cases when $d=1$, $2$, and $3$.

For $d=1$, this means that there exists a sequence $\{g_n\}_{n\in\mathbb{N}}$ in $G$ and   $z_1\in X$ such that
\[
g_nx\rightarrow z_1\quad \text{and}\quad (g_n)^{-1}z_1\rightarrow z_0=y ,
\]
which is exactly what Veech proved in \cite{VAW68}.
See the following figure.
\[
x \overset{\small g_n}{\longrightarrow} z_1 \overset{\small (g_n)^{-1}}{\longrightarrow} z_0=y
\]

\medskip
For $d=2$, this means that there exists a sequence $\{\vec{g}_n=(g_n^1,g_n^2)\}_{n\in\mathbb{N}}$ in $G^2$ and $z_{(1,0)},z_{(0,1)},$ $ z_{(1,1)}\in X$ such that
$$g_n^1x\rightarrow z_{(1,0)},\  g_n^2x\rightarrow z_{(0,1)},\  (g_n^1g_n^2)x\rightarrow z_{(1,1)},$$
and
$$(g_n^1)^{-1}z_{(1,1)}\rightarrow z_{(0,1)}, \ (g_n^2)^{-1}z_{(1,1)}\rightarrow z_{(1,0)}, \ (g_n^1g_n^2)^{-1}z_{(1,1)}\rightarrow z_{(0,0)}=y.$$
See the following figure.
\begin{center}
\quad \quad
\xymatrix{
     & \ z_{(1,0)} \ &  \\
  x\   \ar[ur]^{{g_n^1}} \ar[dr]_{{g_n^2}} \ar[r]^{ g_n^1g_n^2} & \ z_{(1,1)}\  \ar[d]^{(g_n^1)^{-1}} \ar[r]^{(g_n^1g_n^2)^{-1}} \ar[u]_{(g_n^2)^{-1}}  & \ z_{(0,0)}=y \\
   & \  z_{(0,1)}\  &   }
   \end{center}

\medskip

For $d=3$, this means that there exists a sequence $\{\vec{g}_n=(g_n^1,g_n^2,g_n^3)\}_{n\in\mathbb{N}}$ in $G^3$ and
 $z_{(1,0,0)},z_{(0,1,0)},z_{(0,0,1)},z_{(1,1,0)}, z_{(1,0,1)}   ,z_{(0,1,1)}  ,z_{\vec{1}}= z_{(1,1,1)}\in X$ such that
$$g_n^1x\rightarrow z_{(1,0,0)},\  g_n^2x\rightarrow z_{(0,1,0)},\  g_n^3x\rightarrow z_{(0,0,1)},\ g_n^1g_n^2x\rightarrow z_{(1,1,0)},$$
$$ g_n^2g_n^3x\rightarrow z_{(0,1,1)},\ g_n^1g_n^3x\rightarrow z_{(1,0,1)},\ g_n^1g_n^2g_n^3x\rightarrow z_{(1,1,1)},$$
and
$$(g_n^1)^{-1}z_{\vec{1}}\rightarrow z_{(0,1,1)}, \ (g_n^2)^{-1}z_{\vec{1}}\rightarrow z_{(1,0,1)}, \ (g_n^3)^{-1}z_{\vec{1}}\rightarrow z_{(1,1,0)},
(g_n^1g_n^2)^{-1}z_{\vec{1}}\rightarrow z_{(0,0,1)},$$
$$ (g_n^2g_n^3)^{-1}z_{\vec{1}}\rightarrow z_{(1,0,0)},\ (g_n^1g_n^3)^{-1}z_{\vec{1}}\rightarrow z_{(0,1,0,)},\ (g_n^1g_n^2g_n^3)^{_1}z_{\vec{1}}\rightarrow z_{(0,0,0)}=y.$$

\bigskip

The structure of the paper is organized as follows.  
In Section 2, the basic notions and results used in the paper are introduced.
In Section 3, we present a proof of our main result (\cref{main-thm}).  

\medskip

\noindent {\bf Acknowledgments.}
The authors would like to express their gratitude to Professors Wen Huang and Song Shao, as well as Dr. Hui Xu, for their insightful discussions and remarks. 
The authors also extend their appreciation to the anonymous reviewer for providing valuable comments.
The final version of this paper was completed during the second author's visit to the International Center for Mathematics at SUSTech, and he is thankful to the center for its hospitality throughout his stay.
The first author is supported by NNSF of China (12401243) and USTC Research Funds of the Double First-Class Initiative (YD0010002009).  
The second author is supported by NNSF of China (12031019).

\section{Preliminaries}
In this section we gather definitions and preliminary results that
will be necessary later on.
Let $\mathbb{N}$ and $\mathbb{Z}$ be the sets of all positive integers
and integers respectively.

\subsection{Topological dynamical systems}\

Throughout the paper, $(X,G)$ denotes a \emph{topological dynamical system} (t.d.s. for short),
where $X$ is a compact metric space with a metric $\rho$
and $G$ acts  on it as an abelian group of homeomorphisms.
For $x\in X,\mathcal{O}(x,G)=\{gx:g\in G\}$ denotes the \emph{orbit} of $x$.
A t.d.s. $(X,G)$ is called \emph{minimal} if
every point has a dense orbit in $X$.

\subsection{Dynamical cubespaces}\

Let $X$ be a set and let $d\geqslant 1$ be an integer. We view the element $\epsilon\in\{0, 1\}^d$ as
a sequence $\epsilon=(\epsilon(1),\ldots, \epsilon(d))$, where $\epsilon(i)\in\{0,1\}$ for $1\leqslant i\leqslant d$.
Write $\vec{0}=(0,\ldots,0)\in\{0, 1\}^d $, $\vec{1}=(1,\ldots,1)\in\{0, 1\}^d $
and $\vec{1}-\epsilon=(1-\epsilon(1),\ldots, 1-\epsilon(d))$ for $\epsilon=(\epsilon(1),\ldots, \epsilon(d))\in \{0,1\}^d$.

We denote the set of maps $\{0,1\}^d\to X$ by $X^{[d]}$.
For $\epsilon\in \{0,1\}^d$ and $\mathbf{x}\in X^{[d]}$,
${x}_\epsilon$ will be used to denote the $\epsilon$-component of $\mathbf{x}$.
So any element ${\bf x}\in X^{[d]}$ can be viewed as
${\bf x}=({x}_\epsilon: \epsilon\in \{0,1\}^d).$ For example, when $d=2$, a point ${\bf x}\in X^{[2]}=X^4$
can be written as ${\bf x}=({x}_{00},{x}_{10},{x}_{01},{x}_{11})$.

\medskip

Let $G$ be an abelian group with the unit element $e$.
For $\vec{g}=(g_1,\ldots,g_d)\in G^d$ and $\epsilon\in \{0,1\}^d$,
we define
\[
\vec{g}\cdot \epsilon=\prod_{i=1}^{d}g_i^{\epsilon(i)},
\]
and $h^0=e$ for $h\in G$.

Let $(X,G)$ be a t.d.s. and $d\in \mathbb{N}$.
Let $\mathcal{G}^{[d]}$ be the collection of the elements $S\in G^{[d]}$
that can be written as
\begin{equation*}\label{def-cube}
S  =(g\cdot\prod_{i=1}^{d}g_i^{\epsilon_i}:\epsilon\in \{0,1\}^d),
\end{equation*}
where $g,g_1,\ldots,g_d\in G$. For example, when $d=2$, $\mathcal{G}^{[2]}$ is the subgroup of $G^{[2]}$ generated by
$$\{(g,g,g,g):g\in G\}\cup \{(e,h,e,h):h\in G\}\cup \{(e,e,t,t):t\in G\}.$$

Let $\mathcal{F}^{[d]}$ be the collection of the elements $S\in \mathcal{G}^{[d]}$ with $S_{\vec{0}}=e$.
For example, when $d=2$, $\mathcal{F}^{[2]}$ is the subgroup of $G^{[2]}$ generated by
$$ \{(e,h,e,h):h\in G\}\cup \{(e,e,t,t):t\in G\}.$$

For $x\in X$, we write $x^{[d]}=(x,\ldots,x)\in X^{[d]}$.
Let 
\[
\mathbf{Q}^{[d]}(X)=\overline{\{ S x^{[d]}:x\in X,S\in \mathcal{F}^{[d]}\}}.
\]
We call this set the \emph{dynamical cubespace of dimension $d$} of the t.d.s. $(X,G)$.
For convenience, we denote the orbit closure of $\mathbf{x}\in X^{[d]}$
under $\mathcal{F}^{[d]}$ by $\overline{\mathcal{F}^{[d]}}(\mathbf{x})$,
instead of $\overline{\mathcal{O}(\mathbf{x},\mathcal{F}^{[d]})}$.

\medskip

We need the following result from \cite{SY12}.

\begin{theorem}\label{cube-minimal}
  Let $(X,G)$ be a minimal system and $d\in \mathbb{N}$. Then
  \begin{enumerate}
  \item $(\mathbf{Q}^{[d]}(X),\mathcal{G}^{[d]})$ is a minimal system.
    \item $(\overline{\mathcal{F}^{[d]}}(x^{[d]}),\mathcal{F}^{[d]})$ is minimal for all $x\in X$.
  \end{enumerate}
\end{theorem}

\subsection{Regional proximality of higher order}\

\begin{definition}\label{def-rp}\cite{HKM}
Let $(X,G)$ be a t.d.s. and $d\in \mathbb{N}$.
   The \emph{regionally proximal relation of order $d$} is the relation $\textbf{RP}^{[d]}$
defined by: $(x,y)\in\textbf{RP}^{[d]}$ if
and only if for every $\delta>0$, there exist $x',y'\in X$ and $\vec{g}\in G^d$ such that:
$\rho(x,x')<\delta,\rho(y,y')<\delta$ and for every $ \epsilon\in \{0,1\}^d\backslash\{ \vec{0}\}$
\[
\rho(  (\vec{g}\cdot\epsilon) x', (\vec{g}\cdot\epsilon)  y')<\delta.
\]

\end{definition}

It turns out that $\mathbf{RP}^{[d]}$ is closely related to the dynamical cubespaces as the following results indicate.

\begin{theorem}\cite{SY12}\label{rp-equivalence}
Let $(X,G)$ be a minimal system and $d\in \mathbb{N}$.
Then the following statements are equivalent.
\begin{enumerate}
\item $(x,y)\in \mathbf{RP}^{[d]}$;
\item  $(x,y^{[d+1]}_*):=(x,\underbrace{y,y,\ldots,y}_{2^{d+1}-1\ \mathrm{times}})\in \mathbf{Q}^{[d+1]}(X)$;
\item $(x,y^{[d+1]}_*)\in \overline{\mathcal{F}^{[d+1]}}(x^{[d+1]})$.
\end{enumerate}
\end{theorem}

From \cref{rp-equivalence},
we can easily get the following corollary.

\begin{cor}\label{key}
Let $(X,G)$ be a minimal system, $d\in \mathbb{N}$ and $(x,y)\in \mathbf{RP}^{[d]}$.
Then $(x^{[d]},y,x^{[d]}_*)\in \overline{\mathcal{F}^{[d+1]}}(x^{[d+1]})$.
That is,
let $\xi\in \{0,1\}^{d+1}$ such that $\xi(1)=\cdots=\xi(d)=0,\xi(d+1)=1$,
and let $\mathbf{x}=(x_\epsilon:\epsilon\in \{0,1\}^{d+1})\in X^{[d+1]}$
such that $x_{\xi}=y$ and $x_\epsilon=x$ for $\epsilon\in\{0,1\}^{d+1}\backslash\{\xi\}$.
Then $\mathbf{x}\in \overline{\mathcal{F}^{[d+1]}}(x^{[d+1]})$.
\end{cor}

\begin{proof}
As $(x,y)\in \mathbf{RP}^{[d]}\subseteq \mathbf{RP}^{[d-1]}$,
by \cref{rp-equivalence} we get $(x,y^{[d+1]}_*)\in \overline{\mathcal{F}^{[d+1]}}(x^{[d+1]})$
and $(x,y^{[d]}_*)\in \overline{\mathcal{F}^{[d]}}(x^{[d]})$.
It follows from \cref{cube-minimal} (2) that
$(\overline{\mathcal{F}^{[d]}}(x^{[d]}),\mathcal{F}^{[d]})$ is minimal, and thus there exists a sequence $\{\vec{g}_n\}_{n\in \mathbb{N}}\subseteq G^d$
such that $S_n=(\vec{g}_n\cdot \epsilon:\epsilon\in \{0,1\}^d)$
and $\lim\limits_{n\to\infty}S_n(x,y^{[d]}_*)=x^{[d]}$.

Let $\sigma$ be the map from $G^d$ to $G^{d+1}$
given by
\[
\vec{g}=(g_1,\ldots,g_d)\mapsto \sigma(\vec{g})=(g_1,\ldots,g_d,e).
\]
Let $T_n=(\sigma(\vec{g}_n)\cdot \epsilon:\epsilon\in \{0,1\}^{d+1})$ for $n\in \mathbb{N}$.
 Then we have $T_n\in \mathcal{F}^{[d+1]}$
and $\lim\limits_{n\to\infty} T_n(x,y^{[d+1]}_*)=\mathbf{x}$,
    which implies $\mathbf{x}\in \overline{\mathcal{F}^{[d+1]}}(x^{[d+1]})$.
\end{proof}

\section{Proof of \cref{main-thm}}

In this section, we present a proof of our main result. The proof concerning sufficiency is relatively straightforward; however, the proof of necessity is considerably more complex.  
To clarify the concepts involved in the proof of necessity, we will first address the case when $d=2$ and subsequently extend our discussion to encompass the general case.

\begin{proof}[Proof of \cref{main-thm}]
Assume first that there exists a sequence
$\{\vec{g}_n\}_{n\in \mathbb{N}}\subseteq G^d$ and
 $z_{\epsilon}\in X,\epsilon\in \{0,1\}^d$
with $z_{\vec{0}}=y$
such that for every  $\epsilon\in \{0,1\}^d\backslash\{ \vec{0}\}$
\[
\lim_{n\to \infty}(\vec{g}_n\cdot\epsilon)x= z_\epsilon\ \mathrm{and} \
\lim_{n\to \infty}(\vec{g}_n\cdot\epsilon)^{-1}z_{\vec{1}}= z_{\vec{1}-\epsilon}.
\]

We are going to show  $(x,y)\in \mathbf{RP}^{[d]}$. Fix $\delta>0$. Choose $n\in \mathbb{N}$ such that
 for every  $\epsilon\in \{0,1\}^d\backslash\{ \vec{0}\}$
\begin{equation}\label{r1}
\rho((\vec{g}_n\cdot\epsilon)x,z_\epsilon)<\delta\ 
\text{ and}
\
\rho((\vec{g}_n\cdot\epsilon)^{-1}z_{\vec{1}},z_{\vec{1}-\epsilon})<\delta,
\end{equation}
and thus
 for every  $\epsilon\in \{0,1\}^d\backslash\{ \vec{1}\}$
\begin{equation}\label{r2}
\rho((\vec{g}_n\cdot(\vec{1}-\epsilon))^{-1}z_{\vec{1}},z_\epsilon)<\delta.
\end{equation}

Taking $x'=x$ and $y'=(\vec{g}_n\cdot\vec{1})^{-1}z_{\vec{1}}$,
then we have
\begin{align*}
&\rho(x,x')=0, \;\;\rho(y,y')=\rho(z_{\vec{0}},(\vec{g}_n\cdot\vec{1})^{-1}z_{\vec{1}})\overset{\eqref{r1}}{<}\delta,\\
&\rho((\vec{g}_n\cdot\vec{1})x',(\vec{g}_n\cdot\vec{1})y')=
\rho((\vec{g}_n\cdot\vec{1})x,z_{\vec{1}})\overset{\eqref{r1}}{<}\delta,
\end{align*}
and for $\epsilon\in \{0,1\}^d \backslash\{ \vec{0},\vec{1}\}$
\begin{align*}
\rho((\vec{g}_n\cdot\epsilon)x',(\vec{g}_n\cdot\epsilon)y')
&\leqslant  \rho((\vec{g}_n\cdot\epsilon)x,z_\epsilon)+\rho(z_\epsilon,(\vec{g}_n\cdot\epsilon)(\vec{g}_n\cdot\vec{1})^{-1}z_{\vec{1}})\\
&\overset{\eqref{r1}}{<} \delta+\rho(z_\epsilon,(\vec{g}_n\cdot(\vec{1}-\epsilon))^{-1}z_{\vec{1}})\\
&\overset{\eqref{r2}}{<}2\delta,
\end{align*}
which implies $(x,y)\in \mathbf{RP}^{[d]}$ as $\delta$ is arbitrary.

\medskip

We next show the converse.
To make the idea of the proof clearer, we first show the case when $d=2$ and the
general case follows by the same idea.

For $l\in\mathbb{N}$, $(x_1,\ldots,x_l),(y_1,\ldots,y_l)\in X^l$ and $\delta>0$, we write
\[
(x_1,\ldots,x_l)\simeq_{\delta}(y_1,\ldots,y_l)
\]
if $\rho(x_i,y_i)<\delta$ for $1\leqslant i\leqslant l$.
For $\vec{g}=(g_1,\ldots,g_l),\vec{h}=(h_1,\ldots,h_l)\in G^l$, we define
\[
\vec{g}\cdot \vec{h}=(g_1h_1,\ldots,g_lh_l).
\]

\medskip

\noindent {\bf The case $d=2$}.

Let $(x,y)\in \mathbf{RP}^{[2]}$.
Then
\begin{equation}\label{d=2sim}
(x,x,x,x,y,x,x,x)\in \overline{\mathcal{F}^{[3]}}(x^{[3]})
\end{equation}
 by  \cref{key}.

For every finite set $S\subseteq G$ and $\delta>0$, by \eqref{d=2sim} and the uniform continuity of every transformation in $S$
there exist $a,b,c\in G$ such that
\[
(a,b,ab,c,ca,cb,cab)x^7\simeq_{\delta}(x,x,x,y,x,x,x),
\]
where $x^m=(x,\ldots,x)$ ($m$ times), and at the same time
for all $\vec{s}\in S^7$,
\[
\vec{s}  (a,b,ab,c,ca,cb,cab)x^7\simeq_{\delta}    \vec{s} (x,x,x,y,x,x,x).
\]

Let $\{\delta_n\}_{n\in \mathbb{N}}$ be a decreasing sequence of positive real numbers
with $\sum_{n=1}^{\infty}\delta_n<\infty$.
  We first define sequences $\{a_n\}_{n\in \mathbb{N}},\{b_n\}_{n\in \mathbb{N}},\{c_n\}_{n\in \mathbb{N}}\subseteq G$ inductively.

  Choose $a_1,b_1,c_1\in G$ such that
  \[
  (a_1,b_1,a_1b_1,c_1,c_1a_1,c_1b_1,c_1a_1b_1)x^7\simeq_{\delta_1}(x,x,x,y,x,x,x).
  \]
  Now assume that $n\geqslant 2$ and we have already chosen $a_1,\ldots,a_{n-1}$, $b_1,\ldots,b_{n-1}$, $c_1,\ldots,c_{n-1}$.
  Let
  \begin{equation*}
      S_n=\{g\in G:g=\prod_{i=1}^{n-1}a_i^{u_i}\prod_{i=1}^{n-1}b_i^{v_i}\prod_{i=1}^{n-1}c_i^{w_i},u_i,v_i,w_i\in \{0,1\}
  ,i=1,\ldots,n-1\}.
  \end{equation*}
    As $S_n$ is finite, choose $a_{n},b_{n},b_{n}\in G$ such that for all $\vec{s}\in S_n^7$
    \begin{equation}\label{construction}
 \vec{s} (a_{n},b_{n},a_{n}b_{n},c_n,c_{n}a_{n},c_{n}b_{n},c_{n}a_{n}b_{n})x^7
  \simeq_{\delta_{n}}  \vec{s}(x,x,x,y,x,x,x).
  \end{equation}
  This means that for each $s\in S_n$ we have $\rho(sc_nx,sy)<\delta_n$ and
  $$\rho(stx,sx)<\delta_n, \;\forall\; t\in\{a_{n},b_{n},a_{n}b_{n},c_{n}a_{n},c_{n}b_{n},c_{n}a_{n}b_{n}\}.$$
This finishes the inductive definition.

 \medskip

 It is clear that $S_n\subseteq S_{n+1}$, and $a_{n},b_{n},a_{n}b_{n},c_n,c_{n}a_{n},c_{n}b_{n},c_{n}a_{n}b_{n}\in S_{n+1}$ for any $n\in\mathbb{N}$. Moreover, we have $\lim\limits_{n\to\infty}c_nx= y$ since $e\in S_n$.

  Let
  \[
  A_n=\prod_{k=1}^{n}a_{2k-1}b_{2k-1}c_{2k-1}\in S_{2n},\quad B_n=\prod_{k=1}^{n}a_{2k}b_{2k}c_{2k}\in S_{2n+1},
  \]
  and let
 \[
  \alpha_n=a_{2n+1}a_{2n+2}A_n\in S_{2n+3},\quad\beta_n=b_{2n+1}B_n\in S_{2n+2}.
    \]

\medskip

\noindent {\bf Claim 1}:
$\{ \alpha_nx\}_{n\in \mathbb{N}}$ is a Cauchy sequence in $X$.

\begin{proof}[Proof of Claim 1]
For every $n\in \mathbb{N}$, as $a_{2n+1}A_n\in S_{2n+2}$ and $A_n\in S_{2n}\subseteq S_{2n+1}$, we get
 \begin{equation}\label{clm1}
\begin{aligned}
  \rho(\alpha_nx,A_nx)& =  \rho(a_{2n+1}a_{2n+2}A_nx, A_{n}x)\\
    & \leqslant  \rho((a_{2n+1}A_n)a_{2n+2}x,(a_{2n+1}A_n)x)+ \rho(A_na_{2n+1}x,A_nx)\\
  &\overset{\eqref{construction}}{<} \delta_{2n+2}+\delta_{2n+1}<2 \delta_{2n+1},
\end{aligned}
\end{equation}
  and
\begin{align*}
  \rho(\alpha_{n+1}x,A_nx)& \leqslant   \rho(\alpha_{n+1}x,A_{n+1}x)+\rho(A_{n+1}x,A_nx)\\
     &\overset{\eqref{clm1}}{<}2\delta_{2n+3}+ \rho(A_n(a_{2n+1}b_{2n+1}c_{2n+1})x,A_nx)\\
  &\overset{\eqref{construction}}{<}2 \delta_{2n+3}+\delta_{2n+1}<3 \delta_{2n+1}.
  \end{align*}
  
It follows that
\[
  \rho(\alpha_nx,\alpha_{n+1}x) \leqslant  \rho(\alpha_n x, A_{n}x)+ \rho(A_n x,\alpha_{n+1}x)<5\delta_{2n+1}.
\]

 Then for any $m\geqslant n+2$,
  \begin{align*}
     \rho(\alpha_nx,\alpha_{m+1}x) & \leqslant  \rho(\alpha_nx,\alpha_{n+1}x)+\cdots+ \rho(\alpha_mx,\alpha_{m+1}x) \\
      & < 5\sum_{k=n}^{m}\delta_{2k+1},
  \end{align*}
  which implies that $\{ \alpha_nx\}_{n\in \mathbb{N}}$ is a Cauchy sequence in $X$.
\end{proof}

Similarly,
$\{ \beta_nx\}_{n\in \mathbb{N}},\{\alpha_n \beta_nx\}_{n\in \mathbb{N}}$ are also Cauchy sequences in $X$.
Let $z,z_1,z_2\in X$ such that
  $ \lim\limits_{n\to\infty}\alpha_n x= z_1,\lim\limits_{n\to\infty}\beta_n x= z_2
$ and $\lim\limits_{n\to\infty}\alpha_n\beta_n x= z$.

\medskip

\noindent {\bf Claim 2}:
 $ \lim\limits_{n\to \infty}\lim\limits_{m\to \infty}
     \rho(\alpha_n^{-1}\alpha_m\beta_mx,\beta_nx)=0$.

 \begin{proof}[Proof of Claim 2]
    For $m\geqslant n+2$, we have $A_n^{-1}A_m B_n^{-1}B_m=\prod_{k=2n+1}^{2m}(a_kb_kc_k)$. Thus,
  \begin{align*}
    \alpha_n^{-1}\alpha_m\beta_m &=(a_{2n+1}a_{2n+2}A_n)^{-1}a_{2m+1}a_{2m+2}A_m b_{2m+1}B_m\\
  &=a_{2m+2}\cdot D_1\cdot a_{2n+1}^{-1}a_{2n+2}^{-1} (A_n^{-1}A_m B_n^{-1}B_m)\cdot B_n\\
  &=a_{2m+2}\cdot  D_1\cdot D_2\cdot D_3\cdot D_4\cdot B_n,
   \end{align*}
where
\[
D_1=a_{2m+1}b_{2m+1},\;
D_2=\prod_{k=2n+3}^{2m}(a_kb_kc_k),\;
D_3=b_{2n+2}c_{2n+2},\;
D_4=b_{2n+1}c_{2n+1}.
\]

Then
  \begin{align*}
         \rho(\alpha_n^{-1}\alpha_m\beta_mx,B_nx) & =
\rho( a_{2m+2}  D_1D_2D_3D_4 B_nx,B_nx) \\
     & \leqslant \rho(D_1D_2D_3D_4  B_n a_{2m+2}x,  D_1D_2D_3D_4 B_nx)\\
&\quad+\sum_{j=1,2,3} \rho( B_nD_4\cdots D_{j+1} D_j x, B_n D_4\cdots D_{j+1} x)+\rho(B_n D_4 x,  B_nx)\\
&< \delta_{2m+2}+\delta_{2m+1}+\sum_{k=2n+3}^{2m}\delta_k+ \delta_{2n+2}+\delta_{2n+1}\\
&=  \sum_{k=2n+1}^{2m+2}\delta_k,
    \end{align*}
where the term $\sum_{k=2n+3}^{2m}\delta_k$ comes from the estimation of $\rho(B_nD_4D_3D_2x, B_nD_4D_3x)$
by using \eqref{construction} repeatedly. That is, we do it as follows: put $C_n=B_nD_4D_3$ then
 \begin{align*}
 \rho(C_nD_2x, C_nx)&=\rho(C_n \prod_{k=2n+3}^{2m}(a_kb_kc_k) x, C_nx)\\
 &\leqslant \rho(C_n \prod_{k=2n+3}^{2m-1}(a_kb_kc_k)\ a_{2m}b_{2m}c_{2m} x, C_n \prod_{k=2n+3}^{2m-1}(a_kb_kc_k)x)\\
 &\quad \quad \quad \quad \quad \quad \quad +\rho(C_n \prod_{k=2n+3}^{2m-1}(a_kb_kc_k) x, C_nx)\\
 & \overset{\eqref{construction}}{<} \delta_{2m}+ \rho(C_n \prod_{k=2n+3}^{2m-1}(a_kb_kc_k) x, C_nx) \ (\text{as}\  C_n \prod_{k=2n+3}^{2m-1}(a_kb_kc_k)\in S_{2m})\\
 &<\cdots < \sum_{k=2n+3}^{2m}\delta_k.
 \end{align*}

Thus, we have
 \begin{align*}
         \rho(\alpha_n^{-1}\alpha_m\beta_mx,\beta_nx) & \leqslant  \rho(\alpha_n^{-1}\alpha_m\beta_mx,B_nx)
         +\rho(B_nx,b_{2n+1}B_nx)\\
     &<\sum_{k=2n+1}^{2m+2}\delta_k+\delta_{2n+1},
    \end{align*}
 which implies
    $ \lim\limits_{n\to \infty}\lim\limits_{m\to \infty}
     \rho(\alpha_n^{-1}\alpha_m\beta_mx,\beta_nx)=0$.
 \end{proof}

By Claim 2, we get  $\lim\limits_{n\to\infty}\alpha^{-1}_nz= z_2$.
Similarly, we have
 $ \lim\limits_{n\to \infty}\lim\limits_{m\to \infty}
     \rho(\beta_n^{-1}\alpha_m\beta_mx,\alpha_nx)=0$
     and thus $\lim\limits_{n\to\infty}\beta^{-1}_nz= z_1$.

\medskip

\noindent {\bf Claim 3}:
 $ \lim\limits_{n\to \infty}\lim\limits_{m\to \infty}\rho(\alpha_n^{-1}\beta_n^{-1}\alpha_m\beta_mx,c_{2n+1}x)=0$.

\begin{proof}[Proof of Claim 3]
  For $m\geqslant n+2$, by the similar discussion as in the proof of Claim 2 we have
    \[
 \alpha_n^{-1}\beta_n^{-1}\alpha_m\beta_m=   a_{2m+2}\cdot D_1\cdot D_2\cdot D_3\cdot c_{2n+1},
    \]
where
\[
D_1=a_{2m+1}b_{2m+1},\quad
D_2=\prod_{k=2n+3}^{2m}(a_kb_kc_k),\quad
D_3=b_{2n+2}c_{2n+2}.
\]

Then
    \begin{align*}
   &\rho(\alpha_n^{-1}\beta_n^{-1}\alpha_m\beta_mx,c_{2n+1}x)   \\
  = \;&\rho( a_{2m+2}D_1D_2D_3c_{2n+1}x,c_{2n+1}x  )\\
\leqslant\; &\rho( a_{2m+2}D_1D_2D_3c_{2n+1}x,D_1D_2D_3c_{2n+1}x  )+\rho( D_1D_2D_3c_{2n+1}x,D_2D_3c_{2n+1}x  )\\
&\quad\quad\quad \quad +\rho(D_2D_3c_{2n+1}x,D_3c_{2n+1}x  )+\rho(D_3c_{2n+1}x,c_{2n+1}x  )\\
 \overset{\eqref{construction}}{<} \;&\sum_{k=2n+2}^{2m+2}\delta_k,
\end{align*}
which implies
     $ \lim\limits_{n\to \infty}\lim\limits_{m\to \infty}\rho(\alpha_n^{-1}\beta_n^{-1}\alpha_m\beta_mx,c_{2n+1}x)=0$.
\end{proof}

Recall that 
$\lim\limits_{n\to\infty}c_nx= y$.
Thus by Claim 3,
$\lim\limits_{n\to\infty}\alpha^{-1}_n\beta^{-1}_nz= y$.
Taking $z_{(0,0)}=y,z_{(1,0)}=z_1,z_{(0,1)}=z_2,z_{(1,1)}=z$ and $\vec{g}_n=(\alpha_n,\beta_n)$,
we get
  \[
 \lim_{n\to \infty} (\vec{g}_n\cdot\epsilon)x=z_\epsilon\quad \mathrm{and} \quad
   \lim_{n\to \infty} (\vec{g}_n\cdot\epsilon)^{-1}z_{\vec{1}}= z_{\vec{1}-\epsilon},
  \]
for all $\epsilon\in \{0,1\}^2\backslash \{(0,0)\}$.

\bigskip

\noindent {\bf The general case}.

Now we fix $d\in\mathbb{N}$  and assume that $(x,y)\in \mathbf{RP}^{[d]}$.

Let $\xi\in \{0,1\}^{d+1}$ such that $\xi(1)=\cdots=\xi(d)=0$ and $\xi(d+1)=1$.
Let $\mathbf{x}=(x_\epsilon:\epsilon\in \{0,1\}^{d+1})\in X^{[d+1]}$
such that $x_{\xi}=y$ and $x_\epsilon=x$ for $\epsilon\in \{0,1\}^{d+1}\backslash\{\xi\}$.
Then we have $\mathbf{x}\in \overline{\mathcal{F}^{[d+1]}}(x^{[d+1]})$ by \cref{key}.

Let $\{\delta_n\}_{n\in \mathbb{N}}$ be a decreasing sequence of positive numbers
 with $\sum_{n=1}^{\infty}\delta_n<\infty$. Since $\mathbf{x}\in \overline{\mathcal{F}^{[d+1]}}(x^{[d+1]})$,
for every $n\in\mathbb{N}$, 
there is some $\vec{g}\in G^{d+1}$ such that for all $\epsilon\in \{0,1\}^{d+1}$,
\[
\rho((\vec{g}\cdot \epsilon)x,x_\epsilon)<\delta_n.
\]
Moreover, if $H\subseteq G$ is finite, by the uniform continuity of every transformation in $H$
we can assume that  for all $h\in H$ and all $\epsilon\in \{0,1\}^{d+1}$,
\[
\rho(h(\vec{g}\cdot \epsilon)x,hx_\epsilon)<\delta_n.
\]

  We first define a sequence $\{\vec{a}_n\}_{n\in \mathbb{N}}\subseteq G^{d+1}$ inductively.

  Choose $\vec{a}_1\in G^{d+1}$ such that
  for all $\epsilon\in \{0,1\}^{d+1}$,
  \begin{equation*}
    \rho((\vec{a}_1\cdot \epsilon)x,x_\epsilon)<\delta_1.
  \end{equation*}
  Now assume that $n\geqslant 2$ and we have already chosen
  $\vec{a}_1,\ldots,\vec{a}_{n-1}\in G^{d+1}$.
  Let
  \[
H_{n}=\{h\in G:h=\prod_{i=1}^{n-1}(\vec{a}_i\cdot \epsilon_i),\epsilon_i\in \{0,1\}^{d+1}
  ,i=1,\ldots,n-1\}.
  \]
  As $H_{n}$ is finite, we can choose $\vec{a}_n\in G^{d+1}$ such that
  \begin{equation}\label{general}
    \rho(h(\vec{a}_n\cdot \epsilon)x,hx_\epsilon)<\delta_n,
  \end{equation}
  for all $h\in H_{n}$ and all $\epsilon\in \{0,1\}^{d+1}$.

This finishes the inductive definition.

% It is clear that $H_n\subseteq H_{n+1}$, and $\vec{a}_n\cdot \epsilon\in H_{n+1}$ for any $\epsilon\in \{0,1\}^{d+1}$  and  $n\in\mathbb{N}$. 

\medskip

\noindent {\bf Claim 4}:
Let $i_1<\cdots<i_k$ be positive integers and $\epsilon_{i_1},\ldots,\epsilon_{i_k}\in \{0,1\}^{d+1}\backslash\{\xi\}$. 
Then for any $b\in H_{i_1}$ we have 
\[
\rho\big(b\prod_{j=1}^{k}(\vec{a}_{i_j}\cdot{\epsilon_{i_j}})x,bx\big)< k \delta_{i_1}.
\]

\begin{proof}[Proof of Claim 4]
By the construction of every $\vec{a}_{n}$,
we have
\begin{align*}
   & \rho\big(b\prod_{j=1}^{k}(\vec{a}_{i_j}\cdot{\epsilon_{i_j}})x,bx\big) \\
\leqslant\; & \rho\big(b\prod_{j=1}^{k-1}(\vec{a}_{i_j}\cdot{\epsilon_{i_j}} ) (\vec{a}_{i_k}\cdot{\epsilon_{i_k}} ) x,b\prod_{j=1}^{k-1}(\vec{a}_{i_j}\cdot{\epsilon_{i_j}})x\big)+
\cdots+
\rho(b(\vec{a}_{i_1}\cdot{\epsilon_{i_1}})x,bx)\\
\overset{\eqref{general}}{<}\;&\sum_{j=1}^{k}\delta_{i_j}\leqslant k\delta_{i_1},
\end{align*}
as was to be shown.
\end{proof}

  For $n\in \mathbb{N}$, let $\vec{a}_n=(a^{(n)}_1,\ldots,a^{(n)}_{d+1})$.
  For $1\leqslant j\leqslant d$, let
  \[
  A_j^{(n)}=\prod_{i=1}^{n}(\vec{a}_{di-j+1}\cdot \vec{1}),
  \quad
    g_j^{(n)}= A^{(n)}_j\prod_{k=1}^{j}
 a_j^{(dn+k)},
  \]
  and let
\[\vec{g}_n=( g_1^{(n)},\ldots,  g_d^{(n)}).
  \]

\medskip

For $1\leqslant k\leqslant d$ and $\epsilon=(\epsilon(1),\ldots,\epsilon(d))\in \{0,1\}^d$,
let $\theta_k=\theta_k(\epsilon),\beta_k=\beta_k(\epsilon)\in\{0,1\}^{d+1}$ such that
\begin{equation}\label{def-omega}
\theta_k(\epsilon)=(0,\ldots,0,\epsilon(k),\ldots,\epsilon(d),0),
\end{equation}
and
\begin{equation}\label{def-epsilon-tilde}
\beta_k(\epsilon)=(\epsilon(k),\ldots, \epsilon(k)).
\end{equation}

\medskip

For $\epsilon=(\epsilon(1),\ldots,\epsilon(d))\in \{0,1\}^d$,
we have
\begin{flalign}\label{g}
\begin{split}
\vec{g}_n\cdot \epsilon &= \prod_{j=1}^{d}\prod_{k=1}^{j}(a_j^{(dn+k)})^{\epsilon(j)}(A^{(n)}_j)^{\epsilon(j)} \\
&=
\prod_{k=1}^{d}\prod_{j=k}^{d}(a_j^{(dn+k)})^{\epsilon(j)}\cdot
\prod_{j=1}^{d}(A^{(n)}_j)^{\epsilon(j)} \\
&=\prod_{k=1}^{d}(\vec{a}_{dn+k}\cdot \theta_k(\epsilon))
\cdot \prod_{i=1}^n\prod_{j=1}^{d}(\vec{a}_{di-j+1}\cdot \beta_j(\epsilon)).
  \end{split}
\end{flalign}

\medskip

\noindent {\bf Claim 5}:
$\{ (\vec{g}_n\cdot \epsilon)x\}_{n\in \mathbb{N}}$
is a Cauchy sequence in $X$
for every $\epsilon\in \{0,1\}^d\backslash\{\vec{0}\}$.

\begin{proof}[Proof of Claim 5]
Fix $\epsilon=(\epsilon(1),\ldots,\epsilon(d))\in \{0,1\}^d\backslash\{\vec{0}\}$. 
For brevity, let us denote $\theta$ as $\theta(\epsilon)$ and $\beta$ as $\beta(\epsilon)$.
By \eqref{g} we have
\[
\vec{g}_n\cdot \epsilon =\prod_{k=1}^{d}(\vec{a}_{dn+k}\cdot \theta_k)
\cdot \prod_{i=1}^n\prod_{j=1}^{d}(\vec{a}_{di-j+1}\cdot \beta_j).
\]

Let $A_n=\prod_{i=1}^n\prod_{j=1}^{d}(\vec{a}_{di-j+1}\cdot \beta_j)$ for $n\in\mathbb{N}$. Then we have
$A_n\in H_{dn+1}$ and
\[
A_{n+1}=A_n\cdot \prod_{j=1}^{d}(\vec{a}_{d(n+1)-j+1}\cdot \beta_j).
\]

Recall that 
\begin{align*}
\theta_k&= \theta_k(\epsilon)=(0,\ldots,0,\epsilon(k),\ldots,\epsilon(d),0), \\
 \beta_k &= \beta_k(\epsilon)=(\epsilon(k),\epsilon(k),\ldots,\epsilon(k)),
\end{align*}
 and $\epsilon\neq\vec{0}$, then
we have $\theta_k\neq \xi$ and $\beta_k\neq \xi$ for any $1\leqslant k\leqslant d$.

It follows from Claim 4 that
\begin{align*}
 \rho((\vec{g}_n\cdot \epsilon)x,A_nx)
   &= \rho\big(A_n\prod_{k=1}^{d}(\vec{a}_{dn+k}\cdot \theta_k)x,A_nx\big)\\
& <d\delta_{dn+1},
\end{align*}
and
\begin{align*}
 \rho(A_nx,A_{n+1}x)
  &=\rho( A_nx,A_n \prod_{j=1}^{d}(\vec{a}_{d(n+1)-j+1}\cdot \beta_j)x)\\
& < d\delta_{dn+1}.
\end{align*}

Therefore, we have
\begin{align*}
& \rho((\vec{g}_n\cdot \epsilon)x,(\vec{g}_{n+1}\cdot \epsilon)x) \\
\leqslant  \; &\rho((\vec{g}_n\cdot \epsilon)x,A_nx)
+\rho(A_nx,A_{n+1}x)
+\rho(A_{n+1}x,(\vec{g}_{n+1}\cdot \epsilon)x)\\
<\;& d\delta_{dn+1}+d\delta_{dn+1}+d\delta_{dn+d+1}\\
< \;&3 d\delta_{dn+1},
\end{align*}
and thus for $m\geqslant n+2$
\begin{align*}
\rho((\vec{g}_n\cdot \epsilon)x,(\vec{g}_{m+1}\cdot \epsilon)x)
  &\leqslant\sum_{k=n}^{m}\rho((\vec{g}_k\cdot \epsilon)x,(\vec{g}_{k+1}\cdot \epsilon)x) \\
&< \sum_{k=n}^{m}3 d\delta_{dk+1},
\end{align*}
which implies that
$\{ (\vec{g}_n\cdot \epsilon)x\}_{n\in \mathbb{N}}$
is a Cauchy sequence in $X$.
\end{proof}

For $\epsilon\in \{0,1\}^d\backslash\{\vec{0}\}$,
let $z_\epsilon=\lim\limits_{n\to \infty}(\vec{g}_n\cdot \epsilon)x$
and let $z_{\vec{0}}=y$.

  \medskip
\noindent {\bf Claim 6}:
$\lim\limits_{n\to \infty}\lim\limits_{m\to \infty}
\rho( (\vec{g}_n\cdot\epsilon)^{-1}(\vec{g}_m\cdot \vec{1}) x,
\vec{g}_n\cdot(\vec{1}-\epsilon)x )=0$
for every $\epsilon\in \{0,1\}^d\backslash\{\vec{1}\}$.
\begin{proof}[Proof of Claim 6]Fix $\epsilon\in \{0,1\}^d\backslash\{\vec{1}\}$.
For $m\geqslant n+2d$, we have
\begin{equation}\label{5}
(\vec{g}_n\cdot\epsilon)^{-1}(\vec{g}_m\cdot \vec{1})=A\cdot B\cdot C\cdot D,
\end{equation}
where
\begin{equation}\label{ct}
A=\prod_{j=1}^{d}\prod_{k=j}^{d}a_k^{(dm+j)}=\prod_{k=1}^{d}\vec{a}_{dm+k}\cdot \theta_k(\vec{1}),\quad
B=\prod_{i=d(n+1)+1}^{dm}(\vec{a}_{i}\cdot \vec{1}),
\end{equation}
and
\begin{equation}\label{fh}
C=\prod_{k=1}^{d}\big(\vec{a}_{dn+k}\cdot(\vec{1}-\theta_k(\epsilon))\big),\quad
D=
\prod_{i=1}^{n}\prod_{j=1}^{d}\big(\vec{a}_{di-j+1}\cdot(\vec{1}- \beta_j(\epsilon))\big).
\footnote{See \eqref{def-omega} and  \eqref{def-epsilon-tilde}  for the definition of $\theta_k$ and $\beta_j$.}
\end{equation}

It is easy to check that $BCD\in H_{dm+1}$,
$CD\in H_{dn+d+1}$
and $D\in H_{dn+1}$.

By \eqref{g}  we have that
\[
\vec{g}_n\cdot(\vec{1}-\epsilon)  =\prod_{k=1}^{d}(\vec{a}_{dn+k}\cdot \theta_k(\vec{1}-\epsilon))
\cdot \prod_{i=1}^n\prod_{j=1}^{d}(\vec{a}_{di-j+1}\cdot \beta_j(\vec{1}-\epsilon)) .
\]

By the definition \eqref{def-epsilon-tilde} of $\beta_j$, we have
 \[
  \beta_j(\vec{1}-\epsilon)=(1-\epsilon(j),\ldots,1-\epsilon(j)),
 \]
which implies $ \beta_j(\vec{1}-\epsilon)=\vec{1}-\beta_j(\epsilon)$.
It follows that
\[
\vec{g}_n\cdot(\vec{1}-\epsilon)  =\prod_{k=1}^{d}(\vec{a}_{dn+k}\cdot \theta_k(\vec{1}-\epsilon))
\cdot D .
\]

Recall  that
\[
\theta_k(\vec{1}-\epsilon)=(0,\ldots,0,1-\epsilon(k),\ldots,1-\epsilon(d),0)\]
 and $\epsilon\neq\vec{1}$, then
we have $\theta_k\neq \xi$ for any $1\leqslant k\leqslant d$ and thus by  Claim 4  
\begin{align*}
 \rho( \vec{g}_n\cdot(\vec{1}-\epsilon) x,Dx )
 &= \rho(D\prod_{k=1}^{d}(\vec{a}_{dn+k}\cdot \theta_k(\vec{1}-\epsilon))  x,Dx)\\
 &<d\delta_{dn+1}.
\end{align*}

It follows from \eqref{5} \eqref{ct} \eqref{fh}  and Claim 4 that
\begin{align*}
   & \rho( (\vec{g}_n\cdot\epsilon)^{-1}(\vec{g}_m\cdot \vec{1}) x,\vec{g}_n\cdot(\vec{1}-\epsilon)x ) \\
 \leqslant\;& \rho(DCBAx,Dx)+\rho(Dx,\vec{g}_n\cdot(\vec{1}-\epsilon)x)\\
\leqslant\; &\rho(DCBAx,DCBx)+\rho(DCBx,DCx)+\rho(DCx,Dx)+d\delta_{dn+1}\\
<\;&d\delta_{dm+1}+\sum_{i=d(n+1)+1}^{dm}\delta_i+2d\delta_{dn+1},
\end{align*}
which implies
$\lim\limits_{n\to \infty}\lim\limits_{m\to \infty}
\rho( (\vec{g}_n\cdot\epsilon)^{-1}(\vec{g}_m\cdot \vec{1}) x,
\vec{g}_n\cdot(\vec{1}-\epsilon)x )=0$.
\end{proof}

\noindent {\bf Claim 7}:
$\lim\limits_{n\to \infty}\lim\limits_{m\to \infty}
\rho( (\vec{g}_n\cdot\vec{1})^{-1}(\vec{g}_m\cdot \vec{1}) x,a_{d+1}^{(dn+1)}x )=0$.

\begin{proof}[Proof of Claim 7]
For $m\geqslant n+2d$, by \eqref{5} we have 
\[
(\vec{g}_n\cdot\vec{1})^{-1}(\vec{g}_m\cdot \vec{1})
=A\cdot B\cdot C,
\]
where
\[
A=\prod_{k=1}^{d}\vec{a}_{dm+k}\cdot \theta_k(\vec{1}),\;
B=\prod_{i=d(n+1)+1}^{dm}(\vec{a}_{i}\cdot \vec{1}),\;
C=\prod_{k=1}^{d}\big(\vec{a}_{dn+k}\cdot(\vec{1}-\theta_k(\vec{1}))\big).
\]

Recall that 
\[
\theta_k(\vec{1})=(\underbrace{0,\ldots,0}_{k-1},\underbrace{1,\ldots,1}_{d-k+1},0),
\]
then
we have $\vec{1}-\theta_1(\vec{1})=\xi$ and $\vec{1}-\theta_k(\vec{1})\neq \xi$ for $2\leqslant k\leqslant d$ 
which implies
$\vec{a}_{dn+1}\cdot(\vec{1}-\theta_1(\vec{1}))=a_{d+1}^{(dn+1)}$.
Taking $C'=\prod_{k=2}^{d}\big(\vec{a}_{dn+k}\cdot(\vec{1}-\theta_k(\vec{1}))\big)$,
it follows from Claim 4 that
\begin{align*}
 & \rho( (\vec{g}_n\cdot\vec{1})^{-1}(\vec{g}_m\cdot \vec{1}) x,a_{d+1}^{dn+1}x ) \\
 =\;& \rho(ABC' a_{d+1}^{(dn+1)}x,a_{d+1}^{(dn+1)}x)\\
 \leqslant\;&      \rho(a_{d+1}^{(dn+1)}C'BA x,a_{d+1}^{(dn+1)}C'Bx)+ \rho(a_{d+1}^{(dn+1)}C'Bx,a_{d+1}^{(dn+1)}C'x)    \\
 &\quad\quad + \rho(a_{d+1}^{(dn+1)}C'x,a_{d+1}^{(dn+1)}x)\\
< \;&d\delta_{dm+1}+\sum_{i=d(n+1)+1}^{dm}\delta_i+(d-1)\delta_{dn+2},
\end{align*}
which implies
$\lim\limits_{n\to \infty}\lim\limits_{m\to \infty}
\rho( (\vec{g}_n\cdot\vec{1})^{-1}(\vec{g}_m\cdot \vec{1}) x,a_{d+1}^{(dn+1)}x )=0$.
\end{proof}

From Claim 6,
$\lim\limits_{n\to \infty}(\vec{g}_n\cdot\epsilon)^{-1}z_{\vec{1}}=z_{\vec{1}-\epsilon}$ for $\epsilon\in \{0,1\}^d\backslash \{\vec{1}\}$.
By the construction of the sequence $\{\vec{a}_n\}_{n\in \mathbb{N}}$,
we have
$\lim\limits_{n\to \infty}(\vec{a}_n\cdot\epsilon)x=x_\epsilon$ for
all $\epsilon\in \{0,1\}^d\backslash\{\vec{0}\}$.
In particular,
$\lim\limits_{n\to \infty}a_{d+1}^{(n)}x=\lim\limits_{n\to \infty}(\vec{a}_n\cdot\xi)x
=x_\xi=y$. Thus by Claim 7,
$\lim\limits_{n\to \infty}(\vec{g}_n\cdot\vec{1})^{-1}z_{\vec{1}}=y=z_{\vec{0}}$ .

This completes the proof.
\end{proof}

\bibliographystyle{amsplain}

\end{document}